\documentclass[12pt,a4paper]{article}
\usepackage[english]{babel}
\usepackage[T1]{fontenc} 
\usepackage{amsmath, amsfonts, amssymb,amsthm,mathrsfs, youngtab,young} 
\usepackage{graphicx} 
\usepackage{url}
\theoremstyle{plain}
\newtheorem{prop}{Proposition}[section]
\newtheorem{theorem}{Theorem}[section]

\newtheorem{lem}{Lemma}[section]

\theoremstyle{remark}
\newtheorem*{rmk}{Remark}

\theoremstyle{definition}
\newtheorem{definition}{Definition}[section]

\DeclareMathOperator{\GL}{GL}

\DeclareMathOperator{\SL}{SL}

\DeclareMathOperator{\Aut}{Aut}
\DeclareMathOperator{\Out}{Out}

\DeclareMathOperator{\ID}{id}

\DeclareMathOperator{\Ker}{Ker}

\DeclareMathOperator{\Res}{Res}

\DeclareMathOperator{\Irr}{Irr}
\DeclareMathOperator{\Bl}{Bl}

\title{Inductive AM condition for the alternating groups in characteristic $2$}
\author{David Denoncin}
\date{}

\begin{document}
\maketitle
\tableofcontents

\setlength{\parskip}{11pt}

\newpage

\abstract{The aim of this paper is to verify the inductive AM condition stated in \cite{Spa} (definition 7.2) for the simple alternating groups in characteristic $2$. Such a condition, if checked for all simple groups in all characteristics would prove the Alperin-McKay conjecture (see \cite{Spa}). We first check the Alperin-McKay conjecture for double cover of symmetric and alternating groups, $\tilde{\mathfrak{S}}_n$ and $\tilde{\mathfrak{A}}_n$. The proof of this will extend known results about blocks of symmetric and alternating groups of a given weight. 
 
The first two parts are notations and basic results on blocks of symmetric, alternating groups and of their double cover that are recalled for convenience. In the third part we determine the number of height zero spin characters of a given block using certain height preserving bijections between blocks of a given weight. Then we determine the number of height zero spin characters in a given block of the normalizer of a defect group. We finish by showing that the inductive AM condition is true for the alternating groups in characteristic $2$, the cases $\mathfrak{A}_7$ and $\mathfrak{A}_6$ being treated separately from the rest.

I would like to thank Marc Cabanes and Britta Sp\"ath for taking the time to discuss all the mathematical issues I encountered while writing this paper.}

\section{General notations}

\subsubsection*{General setup}
In the following $G$ will always denote a finite group and all considered groups will be finite.  The derived subgroup of $G$ is denoted by $G'$ and if $H$ is a subgroup of $G$ then its index is denoted by $[G:H]$. Blocks of a finite group will be $2$-blocks unless otherwise specified and for an introduction to block theory, numerical defect, defect groups and the theory of domination and covering of blocks, our basic reference is \cite{NT}. All references made to \cite{NT} deals with chapter $5$. In general every notion that requires the use of a prime number will be used for the prime number $2$ unless otherwise specified and the prime number will be ommited (Sylow subgroups, cores of partition, etc...). The Brauer correspondence will always be used between blocks of a group $G$ and blocks of the normalizer of a defect group.

\subsubsection*{Group action}
Let $G$ act on a set $X$. If $x\in X$ denote by $x^G$ its orbit and let $G_x$ be the subgroup of $G$ that fixes $x$.

For example $\Aut(G)$ acts naturally on subgroups of $G$ and if $D$ is a subgroup of $D$, $\Aut(G)_D$ will denote the automorphisms of $G$ stabilizing $D$.

\subsubsection*{Central extensions}

The notation $\tilde{G}$ will denote a group such that its center $Z(\tilde{G})$ is a $2$-group. The quotient group $\tilde{G}/Z(\tilde{G})$ will be denoted by $G$ and we will call $\tilde{G}$ a ($2$-)central extension of $G$.

Throughout we deal with the blocks of $ \mathfrak{A}_n$, $\mathfrak{S}_n$ and of their double cover denoted by $ \tilde{\mathfrak{A}}_n$ and $\tilde{\mathfrak{S}}_n$. For details about these groups, see \cite{Sch} and \cite{HoHu}. These are defined as follows :
\begin{displaymath}
  \tilde{\mathfrak{S}}_n=\left\langle z,t_i, 1\leq i\leq n-1 | \begin{array}{c}z^2=1, t_i^2=z \\(t_it_{i+1})^2=z \textrm{ for $1 \leq i \leq n-2$ } \\  (t_it_j)^3=z  \textrm{ for $i,j$ such that $ |i-j|\geq 2$}\\\end{array}\right\rangle
\end{displaymath}
The center of $\tilde{\mathfrak{S}}_n$ is denoted by $Z$ (except in some places in section $5$ and $6$ where $Z$ will denote a central subgroup of a given group), equals $\{1,z\}$ and we have an exact sequence :
\begin{displaymath}
  1 \rightarrow Z \rightarrow \tilde{\mathfrak{S}}_n \stackrel{\pi}{\rightarrow} \mathfrak{S}_n \rightarrow 1
\end{displaymath}

Then define $\tilde{\mathfrak{A}}_n=\pi^{-1}(\mathfrak{A}_n)$. Recall that $\tilde{\mathfrak{A}}_n$ is the universal covering group of $\mathfrak{A}_n$ for $n\neq 6,7$ and $n\geq 5$ (see the end of chapter $2$ in \cite{HoHu}), and that for $n\geq 4$ the derived subgroup of $\tilde{\mathfrak{S}}_n$ is $\tilde{\mathfrak{A}}_n$.

If $G_1$ and $G_2$ are two groups such that $G_1\cap G_2=Z(G_1)=Z(G_2)=Z$, then we denote by $G_1 \stackrel{\sim}{\times} G_2$ the central product of $G_1$ and $G_2$, that is the group $G_1\times G_2 /\Delta Z$ where $\Delta Z=\{(z,z) , z \in Z\}$. We have an exact sequence :

\begin{displaymath}
  1 \rightarrow \Delta Z \rightarrow G_1\times G_2 \rightarrow G_1\stackrel{\sim}{\times} G_2 \rightarrow 1
\end{displaymath}

We will also need the universal covering groups of $\mathfrak{A}_6$ and $\mathfrak{A}_7$ respectively denoted by $6.\mathfrak{A}_6$ and $6.\mathfrak{A}_7$  (see \cite{GLS3}, theorem 5.2.3).

\subsubsection*{Characters and blocks}
Most of the notations are similar to those in \cite{Spa}. A character of a group will be an irreducible complex character unless otherwise specified. The set of irreducible characters of a group $G$ is denoted by $\Irr(G)$ and is endowed with its usual scalar product denoted by $\left\langle,\right\rangle$. The set of characters of one of the blocks $B$ of $G$ is denoted by $\Irr(B)$.

A spin character of a subgroup $G$ of $\tilde{\mathfrak{S}}_n$ containing $Z$ is a character such that $Z$ is not in its kernel.

If $b$ is a block, $\Irr_0(b)$ denotes the set of height zero characters of $b$ and $|\Irr_0(b)|$ its cardinality. Denote by $\Irr_0(G|D)$ the set of all height zero characters of $G$ which lie in some block with defect group $D$. The set of all blocks of $G$ with defect group $D$ is denoted by $\Bl(G|D)$. If $H$ is a subgroup of $G$ and $\chi \in \Irr(G)$, set $\Irr(H|\chi)=\{\mu \in \Irr(H), \left\langle \mu, \Res_H^G(\chi)\right\rangle \neq 0\}$.

\subsubsection*{Partitions}
Here notations are mostly borrowed from \cite{BeOl2} and \cite{GePf}. We will set $\mathscr{P}(n)$ and $\mathscr{D}(n)$ respectively to be the set of partitions and the set of partitions without repetitions (also called bar partitions) of the integer $n$. Next $\mathscr{P}^s(n)$ denotes the self-conjugate partitions of $n$ and $\mathscr{P}^a(n)$ denotes the set  $\mathscr{P}(n)-\mathscr{P}^s(n)$. If $\lambda \in \mathscr{P}(n)$ then denote by $\lambda^{\star}$ its conjugate partition (see \cite{GePf} \S 5.4). Finally $\mathscr{D}^{+}(n)$ denotes the even bar partitions of $n$ and $\mathscr{D}^{-}(n)$ the odd bar partitions of $n$ (see \cite{BeOl2}).

\section{The $2$-blocks of $\tilde{\mathfrak{S}}_n$ and $\tilde{\mathfrak{A}}_n$, and their characters}

\subsection{The $2$-blocks of symmetric and alternating groups}

For details about the representation theory of symmetric and alternating groups, see \cite{JK}, \cite{Br}, \cite{Ols}, \cite{BeOl1} and \cite{CabEn} \S 5.2. For more specific details about links between representation theory of groups and normal subgroups of index $p$ in characteristic $p$, see \cite{NT}.

The characters of $\mathfrak{S}_n$ are labelled by $\mathscr{P}(n)$, and if $\lambda\in \mathscr{P}^s(n)$ then restriction to $\mathfrak{A}_n$ gives two irreducible characters of $\mathfrak{A}_n$ : $\lambda^+$ and $\lambda^-$. But if $\lambda\in \mathscr{P}^a(n)$ then restrictions of both the characters associated to $\lambda$ and to its conjugate partition give the same irreducible character of $\mathfrak{A}_n$.

Blocks of $\mathfrak{S}_n$ are labelled by cores of partitions, which are self-conjugate and of the form $\{k,k-1,...,1\}$ for some integer $k$. If $\lambda$ is a partition of $n$ then its associated character belongs to the block determined by the core of $\lambda$. The core is a partition of $n-2w$ where $w$ is the weight of the block. Then the core (and thus the block) is determined by its weight as there is no more than one core that is a partition of a given integer. We also know the defect groups : if $B$ is a block of weight $w$  of $\mathfrak{S}_n$, then its defect groups are the Sylow subgroups of $\mathfrak{S}_{2w}$.

Now let us describe the blocks of $\mathfrak{A}_n$. Because $[\mathfrak{S}_{n}:\mathfrak{A}_{n}]=2$ each block of $\mathfrak{A}_{n}$ is covered by a unique block of $\mathfrak{S}_{n}$ (see corollary 5.6 in \cite{NT}). Thus we can define the weight of a block of $\mathfrak{A}_{n}$ by the weight of the unique block of $\mathfrak{S}_{n}$ covering it. Now because cores are self-conjugate, $\mathfrak{S}_n$ and $\mathfrak{A}_n$ have the same number of blocks of positive weight and a bijection is given by block covering.  If $\mathfrak{S}_n$ has a block of weight $0$ (\textsl{i.e}, $n$ is a triangular number) then it covers two blocks of $\mathfrak{A}_n$. If $\lambda$ is a partition of $n$, then either it is a core determining a block of weight $0$ and then both the associated characters in $\mathfrak{A}_n$ are alone in their own block, or all associated characters in $\mathfrak{A}_n$ are in the block determined by the core of $\lambda$. The defect groups of a block of weight $w$ of $\mathfrak{A}_n$ are the Sylow subgroups of $\mathfrak{A}_{2w}$.

\subsection{The 2-blocks of the double covers}

For details about the representations theory of $\tilde{\mathfrak{S}}_n$ and $\tilde{\mathfrak{A}}_n$ see \cite{Sch}, \cite{HoHu} and \cite{BeOl1}.

There are two kinds of characters for $\tilde{\mathfrak{S}}_n$ : spin characters and non spin characters. Non spin characters are characters of $\mathfrak{S}_n$ and thus labelled by partitions of $n$.  
Spin characters are labelled by bar partitions, and if $\lambda$ is an odd bar partition then there are two spin characters of $\tilde{\mathfrak{S}}_n$ associated to it : $\lambda^+$ and $\lambda^-$.  
Those plus all the characters associated with even bar partitions are all the spin characters of $\tilde{\mathfrak{S}}_n$.

The next lemma comes from general results on central extensions (see lemma \ref{brauerquotient}) and relates blocks, defect groups and characters of $\tilde{\mathfrak{S}}_n$ and $\mathfrak{S}_n$. 
\begin{lem}\label{lemdom}
  There is a bijection $\tilde{B}\mapsto B$ given by domination of blocks between the blocks of $\tilde{\mathfrak{S}}_n$ and the blocks of $\mathfrak{S}_n$. Moreover if $\tilde{B}$ dominates $B$ and $D$ is the defect of $B$ then there is an exact sequence : $1 \rightarrow Z \rightarrow \tilde{D} \stackrel{\pi}{\rightarrow} D \rightarrow 1$ where $\tilde{D}$ is the defect group of $\tilde{B}$ and $\pi$ is the natural projection $\tilde{\mathfrak{S}}_n\rightarrow \mathfrak{S}_n$.

Moreover we have that $\Irr(B)$ is a subset of $\Irr(\tilde{B})$ via the natural projection $\tilde{\mathfrak{S}}_n \rightarrow \mathfrak{S}_n$, so that $\Irr(B)=\Irr(\tilde{B})\cap \Irr(\mathfrak{S}_n)$.

\end{lem}

\begin{rmk}
  A block of $\tilde{\mathfrak{S}}_n$ contains at least one spin character because of theorem 8.1 in \cite{NT}.
\end{rmk}

So $\tilde{\mathfrak{S}}_n$ has the same number of blocks as $\mathfrak{S}_n$ and those are labelled by cores of partitions of $n$. The weight of a block of $\tilde{\mathfrak{S}}_{n}$ is the weight of the corresponding block of $\mathfrak{S}_{n}$. The distribution of irreductible characters of $\tilde{\mathfrak{S}}_n$ into these blocks is known by \cite{BeOl1}, but we will not need it.

Now let us explain the situation for $\tilde{\mathfrak{A}}_n$. If $\lambda$ is a partition of $n$ associated with a non spin character of $\tilde{\mathfrak{S}}_n$, then it is a character of $\mathfrak{S}_n$ and in fact the restriction to $\tilde{\mathfrak{A}}_n$ corresponds to the restriction to $\mathfrak{A}_n$ so that we know what happens. Next if $\lambda$ is a bar partition, then if it is even it corresponds to one character of $\tilde{\mathfrak{S}}_n$ that splits into two characters upon restriction to $\tilde{\mathfrak{A}}_n$, $\lambda^+$ and $\lambda^-$.
If it is odd then it corresponds to two irreducible characters $\lambda^{\pm}$ of $\tilde{\mathfrak{S}}_n$ that both give the same character upon restriction to $\tilde{\mathfrak{A}}_n$.

Regarding the blocks now, lemma \ref{lemdom} is still true by replacing $\tilde{\mathfrak{S}}_n$ by $\tilde{\mathfrak{A}}_n$ and $\mathfrak{S}_n$ by $\mathfrak{A}_n$. So that we have the same phenomenon between blocks of $\tilde{\mathfrak{S}}_n$ and blocks of $\tilde{\mathfrak{A}}_n$ as there is between $\mathfrak{S}_n$ and $\mathfrak{A}_n$. That is if $\tilde{B}_S$ is a block of $\tilde{\mathfrak{S}}_n$ of weight $w$, if $w\geq 1$ then it covers a unique block $\tilde{B}_A$ of $\tilde{\mathfrak{A}}_n$ and $\tilde{B}_S$ is the unique block covering $\tilde{B}_A$. But if $w=0$ then the block $\tilde{B}_S$ covers exactly two blocks of $\tilde{\mathfrak{A}}_n$. We define the weight of a block of $\tilde{\mathfrak{A}}_{n}$ by the weight of the unique block of $\tilde{\mathfrak{S}}_{n}$ covering it, which is also the weight of the corresponding block of $\mathfrak{A}_{n}$.

\section{Height zero spin characters in a block $\tilde{B}$}

In this section we use existing results to compute the number of height zero spin characters in a block $\tilde{B}$ of $\tilde{\mathfrak{S}}_n$ or $\tilde{\mathfrak{A}}_n$. We first deal with the principal block which is an easy case. Then obtain a result about blocks of same weight of double covers of alternating groups, which bring us back to considering principal blocks.

\subsection{The case of a principal block}
\label{sec:case-principal-block}

Let me recall a simple lemma with a direct application that immediately settles the case of principal blocks and will ensure that condition 7.1 of the definition 7.2 of \cite{Spa} are satisfied in the cases of maximal defect, the proof of which being fairly easy.

\begin{lem}\label{simplelem}
  Let $G$ be a finite group and $p$ be a prime number. Let $Z(G)_p$ be the set of $p$-elements of the center of $G$. Let $\chi$ be an irreducible character of $G$ with $\chi(1)$ prime to $p$ (\textsl{e.g.} $\chi$ is a height zero character of a block of maximal defect). Then $G'\cap Z(G)_p \subset \Ker(\chi)$. 

In particular there are no height zero spin characters in the principal block of $\tilde{\mathfrak{S}}_{2w}$ or $\tilde{\mathfrak{A}}_{2w}$ for $w\geq 2$. 
\end{lem}
\begin{proof}
  The first part uses the determinant and the fact that central elements acts diagonally. The second part follows because the derived subgroup of $\tilde{\mathfrak{S}}_{2w}$ is $\tilde{\mathfrak{A}}_{2w}$  and the latter always contains $Z$.
\end{proof}

\subsection{General case}
\label{sec:general-case}

We have in general a height preserving bijection between the irreducible characters of a block $B$ in $\mathfrak{S}_n$ of weight $w$ and irreducible characters in the principal block $B_0$ of $\mathfrak{S}_{2w}$, and more generally between two blocks of same weight (see \cite{Eng}, note that this is true in any characteristic). Bessenrodt and Olsson extended this bijection to blocks of $\tilde{\mathfrak{S}}_{n}$ in \cite{BeOl1} and proved that it is height preserving in \cite{BeOl2} (this is a characteristic $2$ feature). This requires some theory of the combinatorics lying behind the representation theory of $\tilde{\mathfrak{S}}_n$ in characteristic $2$ that will not be necessary to expose. Now the main point here is to show that these bijections can be "restricted" to characters of blocks of $\tilde{\mathfrak{A}}_n$ and $\tilde{\mathfrak{A}}_{2w}$.
\begin{theorem}
  Let $\tilde{B}$ be a block of $\tilde{\mathfrak{A}}_n$ of weight $w$, and $\tilde{B}_0$ be the principal block of $\tilde{\mathfrak{A}}_{2w}$. Then there exists a height preserving bijection between characters of $\tilde{B}$ and of $\tilde{B}_0$. This bijection can be chosen to send spin characters to spin characters and non spin characters to non spin characters.
\end{theorem}
\begin{proof}
The bijection constructed by Enguehard in \cite{Eng} is a bijection between partitions of $n$ of weight $w$ with a given core, and partitions of $2w$ of empty core. It is given by associating partitions with same quotient (see \cite{Ols} for the details). Moreover, by proposition 3.5 in \cite{Ols}  and the fact that cores are self-conjugate, we have that this bijection commutes with conjugation on partitions. This means that associating partitions with same quotient gives a bijection between non spin characters of $\tilde{B}$ and of $\tilde{B}_0$. The fact that it is height preserving is true by observing that this operation gives a height preserving bijection between characters of blocks of $\mathfrak{S}_n$, and if a partition is self-conjugate then the height of its associated characters of $\mathfrak{A}_n$ is the same whereas it is increased by one otherwise.

To extend this bijection to spin characters, note that the bijection constructed in \cite{BeOl1} comes from an operation on bar partition that is sign preserving (proposition 3.3 in \cite{BeOl1}), giving thus a bijection between spin characters of $\tilde{B}$ and of $\tilde{B}_0$. The fact that it is height-preserving is justified as before using \cite{BeOl2}, theorem 3.1.
\end{proof}

\begin{rmk}
We now have a similar situation as the general one for symmetric groups in all characteristics : there is a height preserving bijection between blocks of same weight which comes from associating partitions with same quotient (the definition of which depending on the type of partition we are dealing with). The fact that this stays true for the alternating groups and their central extension is a characteristic $2$ feature, and the bijection should give rise to a perfect isometry. In fact blocks of alternating groups of same positive weight in characteristic $2$ have recently been proved to be perfectly isometric by O. Brunat and J.B. Gramain in \cite{BruGr}.
\end{rmk}

In the end we get the following by combining proposition 3.3 in \cite{BeOl1}, theorem 3.1 in \cite{BeOl2} and  the previous theorem : 
\begin{theorem}\label{heightpreserving}
  Let $\tilde{B}$ be a block of weight $w$ of $\tilde{\mathfrak{S}}_{n}$, and $\tilde{B}_0$ be the principal block of $\tilde{\mathfrak{S}}_{2w}$. Then there is a height preserving bijection between the characters of $\tilde{B}$ and those of $\tilde{B}_0$ which preserves the property of being a spin or a non spin character.

The same is true with $\tilde{\mathfrak{A}}_n$.
\end{theorem}

Now we have the number of spin characters of height zero in a given block :

\begin{prop}\label{heightzerobis}
  Let $\tilde{B}$ be a block of weight $w$ of $\tilde{\mathfrak{S}}_{n}$. Then the number of spin characters of height zero in $\tilde{B}$ is :
    \begin{itemize} 
   \item $1$ if $\tilde{B}$ is of weight $0$ 
   \item $2$ if $\tilde{B}$ is of weight $1$ 
   \item $0$ otherwise
  \end{itemize}
  Let $\tilde{B}$ be a block of weight $w$ of $\tilde{\mathfrak{A}}_{n}$. Then the number of spin characters of height zero in $\tilde{B}$ is :
    \begin{itemize} 
   \item $1$ if $\tilde{B}$ is of weight $0$ or $1$ 
   \item $0$ otherwise
  \end{itemize}

\end{prop}
\begin{proof}
We use theorem \ref{heightpreserving}.

If $w\geq 2$ then there are no spin characters of height zero in the principal block of $\tilde{\mathfrak{S}}_{2w}$ or of $\tilde{\mathfrak{A}}_{2w}$ by lemma \ref{simplelem}.

If $w=1$ then $\tilde{\mathfrak{S}}_{2w}=\mathbb{Z}/4\mathbb{Z}$ and $\tilde{\mathfrak{A}}_{2w}=\mathbb{Z}/2\mathbb{Z}$ and the result follows. 

If $w=0$ the result is obvious because $\tilde{\mathfrak{S}}_0=\tilde{\mathfrak{A}}_0=\mathbb{Z}/2\mathbb{Z}$.
\end{proof}

\begin{rmk}
  The case of central defect group is an easy case that can be dealt with using theorem 8.14 in \cite{NT}.
\end{rmk}

\section{Height zero spin characters in $\tilde{b}$}

In this section we calculate the number of height zero spin characters in a block of the normalizer of a given defect group of a group $G \in \{\tilde{\mathfrak{S}}_n, \tilde{\mathfrak{A}}_n\}$. For the reader's convenience we will first recall some lemmas dealing with covering of blocks, domination and the Brauer correspondence. The case $G=\tilde{\mathfrak{S}}_n$ is dealt with using the structure of the normalizer of a defect group of a symmetric group and the nice relationships between domination of blocks and the Brauer correspondence (see lemma \ref{brauerquotient}). The case $G=\tilde{\mathfrak{A}}_n$ is dealt with using more technical results as the normalizer is not so nice. But it is seized between two nice enough groups and gives a situation where we have three groups each of index two in the next, and we conclude using the relationships between the Brauer correspondence, the covering of blocks and the domination of blocks (see lemmas \ref{covering} and \ref{brauerquotient}).

\subsection{Block covering, domination of blocks}
\label{sec:block-cover-domin}

\begin{lem}\label{covering}
  Let $G$ be a finite group. Let $H$ be a normal subgroup of $G$ such that $G/H$ is a $2$-group of order $2^{m}$ and let $b$ be a block of $H$ with defect group $D$ of order $2^{d}$. Then :
 \begin{itemize}
  \item there exists a unique block $B$ of $G$ covering $b$.
  \item  $G$ acts transitively on the defect groups of $b$ so that $[N_{G}(D):N_{H}(D)]=2^{m}$
  \end{itemize}
Moreover if $B$ is of numerical defect $e$, then $e=d+m$ if and only if $b$ and $B$ correspond to equal block idempotents.
\end{lem}
\begin{proof}
  This is a direct consequence of corollary 5.6 and theorem 5.16 of \cite{NT}.
\end{proof}

Let us give a general lemma relating blocks of a group and of a $2$-central extension.
\begin{lem}\label{brauerquotient}
 Let $\tilde{G}$ be a finite group with center $Z$, a $2$-group. Define $\pi : \tilde{G} \rightarrow G=\tilde{G}/Z$, the canonical quotient map.

Then there is a bijection given by domination of blocks between the blocks of $\tilde{G}$ and the blocks of $G$. Moreover if $\tilde{B}$ is a block with defect $\tilde{D}$ of $\tilde{G}$ dominating $B$, then $B$ has defect group $D=\pi(\tilde{D})$ and domination gives a bijection between blocks of $N_{\tilde{G}}(\tilde{D})$ and of $N_{G}(D)$. If $\tilde{b}$ and $b$ are the Brauer correspondents of $\tilde{B}$ and $B$, then we have that $\tilde{b}$ is the block of $N_{\tilde{G}}(\tilde{D})$ dominating $b$ in $N_{G}(D)$.
\end{lem}
\begin{proof}
The first part is readily obtained from theorems 8.10 and 8.11 of \cite{NT}. For the second part let us borrow the notations from \cite{NT}. Then by lemma 8.5  and theorem 8.11 of \cite{NT},  we need to prove that $w_{\tilde{b}}=w_{b}\circ \mu_{Z}^{*}$. The hypothesis gives $w_{\tilde{b}}\circ Br_{\tilde{D}}=w_{b}\circ Br_{D}\circ \mu_{Z}^{*}$. But $Br_{D}\circ \mu_{Z}^{*}=\mu_{Z}^{*}\circ Br_{\tilde{D}}$. Now again using theorem 8.11 for normalizers with corollary 1.12 and lemma 2.13 (ii) of \cite{NT} we get the result.
\end{proof}

\subsection{Calculation for normalizers in $\tilde{\mathfrak{S}}_n$}

This subsection is devoted to calculating the number of height zero spin characters in the Brauer correspondent of a given block $\tilde{B}$ of $\tilde{\mathfrak{S}}_{n}$. The idea is to reduce the problem to a calculation within a principal block as before. Technically the proof will go as follows : we will deal with the case $w=0$, and then if $w\geq 1$  we will deal first with the case $w=2n$ and use it to treat the cases $n-2w>0$.

The simple case of blocks of weight $0$ have been settled before : it is the case of blocks with defect group $Z$. 

Now let $\tilde{B}_{S}$ be a block of weight $w\geq 1$ of $\tilde{\mathfrak{S}}_{2w}$ (this is necessarily the principal block). Let $\tilde{D}$ be a defect group of $\tilde{B}_{S}$, that is a Sylow subgroup of $\tilde{\mathfrak{S}}_{2w}$.

\begin{lem}\label{particularcase}
The group $\tilde{N}_{\tilde{\mathfrak{S}}_{2w}}(\tilde{D})$ has only one block.

If $w\geq 2$ then this block has no spin characters of height zero. 

If $w=1$ then all its characters are of height zero, two of which being spin characters.
\end{lem}
\begin{proof}
If $w\geq 3$ then we have $\tilde{N}_{\tilde{\mathfrak{S}}_{2w}}(\tilde{D})=\tilde{D}$ (see \cite{MichOl} \S 3) which has only one block, and no spin characters of height zero because of lemma \ref{simplelem} ($\tilde{D}$ is a non abelian $2$-group and $Z$ is a simple group contained in $\tilde{D}$  (see theorem 2.8 of \cite{NT})).

If $w=2$ then $\tilde{N}_{\tilde{\mathfrak{S}}_{2w}}(\tilde{D})=\tilde{\mathfrak{S}}_{4}=\GL_2(3)$ which has only one block and no spin characters of height zero because of lemma \ref{simplelem}  (or by looking at the character table).

If $w=1$ then $\tilde{\mathfrak{S}}_{2w}=\tilde{N}_{\tilde{\mathfrak{S}}_{2w}}(\tilde{D})=\mathbb{Z}/4\mathbb{Z}$ and the statement is clear.
\end{proof}

Now let $\tilde{B}_{S}$ be a block of $\tilde{\mathfrak{S}}_n$ with defect group $\tilde{D}$ and weight $w\geq 1$ (we are not in the case of central defect) such that $n-2w>0$, dominating the block $B_{S}$ with defect $D$. Let $\tilde{b}_{S}$ be the Brauer correspondent of $\tilde{B}_{S}$ in $N_{\tilde{\mathfrak{S}}_n}(\tilde{D})$. We have $N_{\mathfrak{S}_n}(D)=N_{\mathfrak{S}_{2w}}(D)\times \mathfrak{S}_{n-2w}$, and the Brauer correspondent of $B_{S}$ is $b_{S}=b_0\times b_{1}$, the product of the principal block $b_0$ of $N_{\mathfrak{S}_n}(D)$ and of the unique block $b_{1}$ of numerical defect $0$ (and thus of weight $0$) in $\mathfrak{S}_{n-2w}$ (see the proof of lemma 11.3 of \cite{Ols} and \S 5.2 in \cite{CabEn}).

\begin{lem}
  Let $\tilde{b}_0$ be the principal block of $N_{\tilde{\mathfrak{S}}_{2w}}(\tilde{D})$, and $\tilde{b}_{S}$ be the unique block of $\tilde{\mathfrak{S}}_{n-2w}$ that dominates $b_{S}$. We have that $\tilde{b}_{1}$ is a block with defect group $Z$.

 Then  $\tilde{b}_{S}$ is the block dominated by $\tilde{b}_0\times \tilde{b}_{1}$.
\end{lem}
\begin{proof}
The block $\tilde{b}_{1}$ is of defect group $Z$ because of lemma \ref{brauerquotient}.

Then $N_{\tilde{\mathfrak{S}}_n}(\tilde{D})$ is the central product of $N_{\tilde{\mathfrak{S}}_{2w}}(\tilde{D})$ and $\tilde{\mathfrak{S}}_{n-2w}$ so that this lemma makes sense. Now lemma \ref{brauerquotient} tells us that $\tilde{b}_{S}$ is the block dominating $b_{S}=b_0\times b_{1}$. Moreover  $\tilde{b}_0\times \tilde{b}_{1}$ dominates $b_0 \times b_{1}$, so  we can conclude that indeed $\tilde{b}_0\times \tilde{b}_{1}$ dominates $\tilde{b}_{S}$ by lemma \ref{brauerquotient}.
\end{proof}

This discussion leads to the following lemma :
\begin{lem}\label{caseofc}
 Let $w,n$ be integers such that $w\geq 1$ and $n-2w>0$. Let $\tilde{B}_{S}$ be a block of weight $w$ of $\tilde{\mathfrak{S}}_{n}$, with defect group $\tilde{D}$. Let $\tilde{b}_{S}$ be the Brauer correspondent of $\tilde{B}_{S}$ in $N_{\tilde{\mathfrak{S}}_n}(\tilde{D})$. If $w\geq 2$ then all height zero characters of $\tilde{b}_{S}$ are non spin. If $w=1$ then $\tilde{b}_{S}$ has two spin characters of height zero.
 
\end{lem}
\begin{proof}
Keep the notations of the previous lemma.
To get a character of height zero in $\tilde{b}_{S}$ we need to take one of height zero in $N_{\tilde{\mathfrak{S}}_{2w}}(\tilde{D})$ and one of $\tilde{b}_{1}$ and such that the product has $\Delta Z$ in its kernel (see \cite{NT} theorem 8.6). But $\tilde{b}_{1}$ has only two characters both of height zero, one spin and one non spin.

Then the result follows from lemma \ref{particularcase}.
\end{proof}

So in the end we obtain the number of height zero spin characters :
\begin{prop}\label{heightzerospin}
Let $\tilde{B}_{S}$ be a block of weight $w$ of $\tilde{\mathfrak{S}}_{n}$, with defect group $\tilde{D}$. Let $\tilde{b}_{S}$ be the Brauer correspondent of $\tilde{B}_{S}$ in $N_{\tilde{\mathfrak{S}}_n}(\tilde{D})$. Then the number of height zero spin characters in $\tilde{b}_{S}$ is :
\begin{itemize}
\item $1$ if $\tilde{B}_{S}$ is of weight $0$ 
\item $2$ if $\tilde{B}_{S}$ is of weight $1$
\item $0$ otherwise
\end{itemize}
\end{prop}
\begin{proof}
  The first case is dealt with in theorem \ref{heightpreserving}. The other cases are dealt with in the previous lemma and lemma \ref{particularcase}.
\end{proof}

\subsection{Calculation for normalizers in $\tilde{\mathfrak{A}}_n$}

The previous arguments are no longer valid because if $D$ is a defect group of a block of weight $w\geq 2$ of $\mathfrak{A}_n$, which is a Sylow subgroup of $\mathfrak{A}_{2w}$, it does not necessarily hold that $N_{\mathfrak{A}_n}(D)=N_{\mathfrak{A}_{2w}}(D)\times \mathfrak{A}_{n-2w}$. But we will adapt the proof using the groups $N_{\mathfrak{A}_{2w}}(D)\times \mathfrak{A}_{n-2w}$ and $N_{\mathfrak{S}_{2w}}(D)\times \mathfrak{S}_{n-2w}$. The strategy is still to reduce the problem to a principal block. Technically the proof will go as follows : treat the easy cases of weight $0$ and $1$, then deal with the particular cases for which $w\geq 2, n=2w$ and  $w\geq 2, n=2w+1$ (those are the cases of principal block) and then deal with the remaining cases $w\geq 2, n-2w>1$ using what has been done before.

From now if $G_{1}$ and $G_{2}$ are two groups having $Z$ as their center, and $b_{1}, b_{2}$ are blocks of $G_{1}$ and $G_{2}$ respectively then we will denote by $b_{1}.b_{2}$ the block of the central product of $G_{1}$ and $G_{2}$ which is dominated by the product of the blocks $b_{1}\times b_{2}$ in the product $G_{1}\times G_{2}$ according to lemma \ref{brauerquotient}.

Note that a defect group of a block of weight $0$ or $1$ is central, so that these cases are easily dealt with.

Now let $\tilde{B}_{A}$ be a block of weight $w\geq 2$ of $\tilde{\mathfrak{A}}_{2w}$. Such a block is the principal block. Let $\tilde{D}$ be a defect group of $\tilde{B}_{A}$, which is a Sylow subgroup of $\tilde{\mathfrak{A}}_{2w}$. 

\begin{lem}\label{principalblockcase}
The group $N_{\tilde{\mathfrak{A}}_{2w}}(\tilde{D})$ has only one block, and it is covered by the unique block of $N_{\tilde{\mathfrak{S}}_{2w}}(\tilde{D})$. In both those blocks, there are no spin characters of height zero. Moreover $N_{\tilde{\mathfrak{S}}_{2w}}(\tilde{D})$ is the normalizer in $\tilde{\mathfrak{S}}_{2w}$ of one of the Sylow subgroups of $\tilde{\mathfrak{S}}_{2w}$ containing $\tilde{D}$.
\end{lem}
\begin{proof}
If $w\geq 3$ then $N_{\tilde{\mathfrak{A}}_{2w}}(\tilde{D})=\tilde{D}$ (see \cite{MichOl} \S 3). It has no height zero spin characters by lemma \ref{simplelem}.

If $w=2$ then $N_{\tilde{\mathfrak{A}}_{2w}}(\tilde{D})=\tilde{\mathfrak{A}}_4=\SL_{2}(3)$ which has only one block and no spin characters of height zero by lemma \ref{simplelem} (or by looking at the character table).

We have by lemma \ref{covering} that $[N_{\tilde{\mathfrak{S}}_{2w}}(\tilde{D}):N_{\tilde{\mathfrak{A}}_{2w}}(\tilde{D})]=2$ and because of the previous description of those normalizers and of lemma \ref{particularcase} we obtain that in fact $N_{\tilde{\mathfrak{S}}_{2w}}(\tilde{D})$ is the normalizer in $\tilde{\mathfrak{S}}_{2w}$ of one of the Sylow subgroups of $\tilde{\mathfrak{S}}_{2w}$ containing $\tilde{D}$.

The fact about the covering of block is then obvious and we know that there are no spin characters of height zero in $N_{\tilde{\mathfrak{S}}_{2w}}(\tilde{D})$ by lemma \ref{particularcase}.
\end{proof}

The block of $N_{\tilde{\mathfrak{A}}_{2w}}(\tilde{D})$ will be denoted by $\tilde{c}_{0}$ and that of $N_{\tilde{\mathfrak{S}}_{2w}}(\tilde{D})$ will be denoted by $\tilde{b}_{0}$. Note that the notation agree with that of the previous section and that those blocks corresponds to equal block idempotents (the trivial one).

Let us treat a similar case.  Let $\tilde{B}$ be a block of weight $w\geq 2$ of $\tilde{\mathfrak{A}}_{2w+1}$. Such a block is again the principal block and a defect group $\tilde{D}$ is a Sylow subgroup of $\tilde{\mathfrak{A}}_{2w}$ which is also a Sylow subgroup of $\tilde{\mathfrak{A}}_{2w+1}$.

\begin{lem}\label{principalblock2w1}
  The group $N_{\tilde{\mathfrak{A}}_{2w+1}}(\tilde{D})$  has only one block, and no spin characters of height zero.
\end{lem}
\begin{proof}
It is the same proof as in the previous proposition.  

If $w\geq 3$ then $N_{\tilde{\mathfrak{A}}_{2w+1}}(\tilde{D})=\tilde{D}$ (see \cite{MichOl} \S 3) has only one block and no spin characters of height zero by lemma \ref{simplelem}.

If $w=2$ then as in the previous proposition we have $N_{\tilde{\mathfrak{A}}_{2w+1}}(\tilde{D})=\tilde{\mathfrak{A}}_4=\SL_2(3)$ which has only one block with no spin characters of height zero.
\end{proof}

Now let us deal with the remaining cases.

Let $\tilde{B}_A$ be a block of $\tilde{\mathfrak{A}}_n$, of weight $w\geq 2$ such that $n-2w>1$, and with defect group $\tilde{D}$ a Sylow subgroup of $\tilde{\mathfrak{A}}_{2w}$. Such a block is covered by a block $\tilde{B}_{S}$ of $\tilde{\mathfrak{S}}_{n}$ whose Brauer correspondent is $\tilde{b}_{S}=\tilde{b}_{0}.\tilde{b}_{1}$ with and $\tilde{b}_{1}$ being the block of weight $0$ of $\tilde{\mathfrak{S}}_{n-2w}$ as in the previous section.  Let $\tilde{b}_{A}$ be the Brauer correspondent of $\tilde{B}_A$ in $\tilde{N}_A=N_{\tilde{\mathfrak{A}}_n}(\tilde{D})$. Then $\tilde{b}_{A}$ is covered by $\tilde{b}_{0}.\tilde{b}_{1}$ thanks to \cite{Nav} theorem 9.28. Moreover by lemma \ref{covering} we have that $\tilde{b}_{A}$ and $\tilde{b}_{0}.\tilde{b}_{1}$ correspond to equal block idempotents.

Let $\tilde{H}=N_{\tilde{\mathfrak{A}}_{2w}}(\tilde{D}) \stackrel{\sim}{\times}  \tilde{\mathfrak{A}}_{n-2w}$ and let $\tilde{N}_S=N_{\tilde{\mathfrak{S}}_n}(\tilde{D})=N_{\tilde{\mathfrak{S}}_{2w}}(\tilde{D})\stackrel{\sim}{\times} \tilde{\mathfrak{S}}_{n-2w}$. We have the inclusions : $\tilde{H}\subset \tilde{N}_A\subset \tilde{N}_S$. Observe that $[\tilde{N}_S:\tilde{N}_A]=2$ by lemma \ref{covering}. Then $\tilde{H}$ is a subgroup of $\tilde{N}_A$ of index $2$ because it is a subgroup of index $4$ in $\tilde{N}_S$ (recall the descriptions of the normalizers in $\tilde{\mathfrak{A}}_{2w}$ and $\tilde{\mathfrak{S}}_{2w}$ and that $n-2w>1$). We have $\tilde{H}\triangleleft \tilde{N}_S$, and $\tilde{H}\triangleleft \tilde{N}_A \triangleleft \tilde{N}_{S}$.

Recall that $\tilde{b}_{1}$ covers two blocks of weight $0$ of $\tilde{\mathfrak{A}}_{n-2w}$, let us denote them by $\tilde{c}_{1}, \tilde{c}'_{1}$. Then $\tilde{b}_{0}.\tilde{b}_{1}$ covers exactly two blocks of $\tilde{H}$ which are $\tilde{c}=\tilde{c}_{0}.\tilde{c}_{1}$ and $\tilde{c}'=\tilde{c}_{0}.\tilde{c}'_{1}$. So that $\tilde{b}_{A}$ covers exactly two blocks of $\tilde{H}$ which are $\tilde{c}$ and $\tilde{c}'$.

\begin{lem}\label{specialcase}
 The blocks $\tilde{c}$ and $\tilde{c}'$ have no spin characters of height zero.
\end{lem}
\begin{proof}
This is the same proof as in lemma \ref{caseofc}, using lemma \ref{principalblockcase}.
\end{proof}

This discussion leads to the following lemma :

\begin{lem}
  Let $w,n$ be integers such that $w\geq 2$ and $n-2w>1$. Let $\tilde{B}_{A}$ be a block of weight $w$ of  $\tilde{\mathfrak{A}}_n$ and let $\tilde{D}$ be a defect group. Let $\tilde{b}_{A}$ be the Brauer correspondent of $\tilde{B}_{A}$ in $N_{\tilde{\mathfrak{A}}_n}(\tilde{D})$.
 Then  all height zero characters of $\tilde{b}_{A}$ are non spin.
\end{lem}
\begin{proof}
 We keep the previous notations. 

 We know that there are no spin characters of height zero in $\tilde{c}$ by lemma \ref{specialcase}, or in $\tilde{c}'$. Let $\chi$ be a height zero character in $\tilde{b}_{A}$. Then its restriction to $\tilde{H}$ is a sum of height zero characters lying in $\tilde{c}$ or $\tilde{c}'$ (by theorem 5.17 in \cite{NT}). Those are all non spin by lemma \ref{specialcase}. Hence $\chi$ cannot be a spin character.
\end{proof}

The next proposition concludes the study of this section :

\begin{prop}\label{heightzero}
  Let $\tilde{B}_{A}$ be a block of $\tilde{\mathfrak{A}}_n$ with defect $\tilde{D}$ and weight $w$, and let $\tilde{b}_{A}$ be its Brauer correspondent in $N_{\tilde{\mathfrak{A}}_n}(\tilde{D})$. Then the number of spin characters of height zero in $\tilde{b}_{A}$ is :
  \begin{itemize} 
   \item $1$ if $\tilde{B}_{A}$ is of weight $0$ or $1$  
   \item $0$ otherwise
  \end{itemize}
\end{prop}
\begin{proof}
  The cases of central defect groups corresponding to $w=0$ or $w=1$ are dealt with using theorem \ref{heightpreserving}. Now if $w\geq 2$ the cases $n=2w$ and $2w=2w+1$ were treated separately in this section in lemmas \ref{principalblockcase} and \ref{principalblock2w1}. The remaining cases are dealt with in the previous lemma.
\end{proof}

\section{Inductive AM condition for $\mathfrak{A}_n$, $n \neq 6,7$}

Before stating the inductive AM condition let us state for convenience an easily adapted lemma from lemma 2.2 in \cite{Spa} :

\begin{lem}
  Let $X$ be a finite group, $Z$ be a subgroup of $Z(X)$, and set $\bar{X}=X/Z$ and $\pi : X\rightarrow \bar{X}$ the canonical surjection. Let $rep : \bar{X} \rightarrow X$ be a $Z$ section, \textsl{i.e} a map such that $rep(1)=1$ and $\pi \circ rep = \ID_{\bar{X}}$. Let $\chi\in\Irr(X)$ and $\rho_{\chi} : X \rightarrow \GL_{\chi(1)}(\mathbb{C})$ be a representation affording $\chi$.

Then the map $\mathscr{P} : \bar{X} \rightarrow GL_{\chi(1)}(\mathbb{C})$ defined by $\mathscr{P}=\rho_{\chi} \circ rep$ is a projective representation (see \cite{Isa}) and its factor set $\alpha$ satisfies :
\begin{displaymath}
  \alpha(x,y)=\mu(rep(x).rep(y).rep(xy)^{-1}) \textrm{ where } \mu\in Irr(Z) \textrm{ and } \chi_{|Z}=\chi(1)\mu
\end{displaymath}

In this case we say that $\mathscr{P}$ is obtained from $\chi$ using $rep$.
\end{lem}

Now let us recall a part of the inductive AM condition from \cite{Spa} (definition 7.2) that will be sufficient for our purpose.

\begin{definition}\label{WeakAM}
   Let $p$ be any prime number. Let $S$ be a perfect simple group, $G$ its universal covering group and $Z=Z(G)$. Let $D$ be a non central defect group of a $p$-block of $G$. We say that the weak AM condition holds for $S$ in characteristic $p$ with respect to $D$ if :
  \begin{enumerate}
  \item There exists an $\Aut(G)_D$-equivariant bijection $\Omega_D : \Irr_0(G|D) \rightarrow \Irr_0(N_G(D)|D)$
  \item $\Irr(Z|\chi)=\Irr(Z|\Omega_D(\chi))$ for every $\chi\in \Irr_0(G|D)$
  \item $\Omega_D(\Irr_0(B))=\Irr_0(b)$ whenever $B$ and $b$ are Brauer correspondent.
  \end{enumerate}
\end{definition}

\begin{rmk}
The actual inductive AM condition stated in definition 7.2 in \cite{Spa} involves conditions on projective representations of certain automorphism groups of $S$, and an equality on trace maps (equation 7.4) which are automatically satisfied if for example $\Aut(S)_D$ is a $p$-group. Those conditions are conditions 7.2(iii) in \cite{Spa}.
\end{rmk}

\begin{definition}
  Keep the notations of definition \ref{WeakAM}. We say that the inductive AM condition is true for $S$ in characteristic $p$ if it holds for any $p$-subgroup which is the defect group of some block of $G$.
\end{definition}

The main result of this paper is the following :

\begin{theorem}\label{InductiveAM}
  The inductive AM condition is true for the simple alternating groups in characteristic $2$.
\end{theorem}

The proof of the theorem is the combination of proposition \ref{indAMn} below in which we deal with all the cases but those of $\mathfrak{A}_6$ and of $\mathfrak{A}_7$, and of propositions \ref{indAM7} and \ref{indAM6} below in which we deal with the remaining cases as they have an exceptional Schur cover.

\subsection{Alperin-McKay conjecture for $\tilde{\mathfrak{S}}_n$ and $\tilde{\mathfrak{A}}_n$}

Everything shown in the last two sections leads to the following result :

\begin{prop}
Let $\tilde{B}$ be a $2$-block of weight $w$ and defect group $\tilde{D}$ of $G\in \{\tilde{\mathfrak{S}}_{n}, \tilde{\mathfrak{A}}_n\}$, and $\tilde{b}$ its Brauer correspondent in $N_G(\tilde{D})$.

Then $|Irr_0(\tilde{B})|=|Irr_0(\tilde{b})|$, that is the Alperin-McKay conjecture is true for $G$. Moreover there exists a bijection between height zero characters that preserves spin characters.
\end{prop}
\begin{proof}
  The Alperin-McKay conjecture is true for $\mathfrak{S}_n$ (see \cite{Ols}) and $\mathfrak{A}_n$ (see \cite{MichOl}). So that we have a bijection between height zero non spin characters of $\tilde{B}$ and $\tilde{b}$. Then we have to deal with the spin characters of $\tilde{B}$. Thus the result is obtained by combining propositions \ref{heightzerobis}, \ref{heightzerospin} and \ref{heightzero}.
\end{proof}

\subsection{Inductive AM condition for $\mathfrak{A}_n, n\neq 6,7$}

Now we can prove a part of the main result of this paper : 
\begin{prop}\label{indAMn}
  The inductive AM condition is true for $\mathfrak{A}_n$ for $n\neq 6,7$ and $n\geq 5$.
\end{prop}
\begin{proof}  
This is done using the proof of lemma 8.1 of \cite{Spa} (recall that $\Out(\mathfrak{A}_n)$ is cyclic simple in our cases). Indeed the condition that the prime number should not divide the order of the center is not met, but in the proof it is used to ensure that the equivariant bijection sends non spin characters to non spin characters and spin characters to spin characters (this is condition $3$ in the definition recalled above). This condition is true in our case because there are no spin characters of height zero in a block of $\tilde{\mathfrak{A}}_n$ of non central defect group and the same is true for Brauer correspondents (see propositions \ref{heightzerobis} and \ref{heightzero}).
\end{proof}

\section{Inductive AM condition for $\mathfrak{A}_{7}$ and $\mathfrak{A}_6$}
\label{sec:induct-am-cond}

\subsection{Simplifying lemmas}

Because the condition 7.2(iii) from definition 7.2 in \cite{Spa} only deal with one character at a time we will prove a "character by character" version of lemma 8.1 of \cite{Spa}. Then we will prove a simple lemma which will ensure that condition $2$ of definition \ref{WeakAM} is automatically satisfied. 

In this subsection the prime number for which everything is defined will not necessarily be the prime number $2$.

\begin{definition}
  Let $G$ be a finite group, $Z$ be its center. Let $\chi\in \Irr(G)$. Denote by $\mu_{\chi}$ the element of $\Irr(Z)$ such that $\chi_{|Z}=\chi(1).\mu_{\chi}$, let $K_{\chi}=\ker(\mu_{\chi})$, $G_{\chi}=G/K_{\chi}$ and let $\pi : G \rightarrow G_{\chi} $ be the canonical surjection. For such $\chi$ denote by $\bar{\chi}$ the element of $\Irr(G_{\chi})$ such that $\bar{\chi} \circ \pi=\chi$.
\end{definition}
Note that we have $\bar{\chi}=\chi\circ rep$ for any $Z$-section.

\begin{lem}\label{lemA6A7}
  Let $p$ be a prime number. Let $S$ be a finite simple group, $G$ its universal covering group and $Z$ its center. Let $rep$ be a $Z$-section and $D$ a defect group of some $p$-block of $G$. Let $\pi : G \rightarrow S$ be the canonical map.

Suppose that the weak AM condition is true for $S$ with respect to $D$, and denote by $\Omega_D$ the $\Aut(G)_D$-equivariant bijection. Now suppose that for all $\chi\in \Irr_0(G|D)$, $\Out(G)_{\chi}=\Aut(S)_{\chi}/S$ is cyclic of prime order or trivial.

Then the inductive AM condition is true for $S$ with respect to $D$.
\end{lem}
\begin{proof}
To see that a character satisfying the hypothesis also satisfies the conditions 7.2(iii) of \cite{Spa} we use the proof of lemma 8.1 of \cite{Spa}. Let $\bar{rep}$ be the composition of $rep : S \rightarrow G$ and of the canonical map $G \rightarrow G_{\chi} $. This is a $Z(G_{\chi})$-section.
Define $O=G_{\chi}\rtimes \Out(S)_{\chi}$ (this is possible because $\Out(G)_{\chi}$ is cyclic and $\Aut(S)_{\chi}$ preserves $K_{\chi}$). Then we have $G_{\chi}\triangleleft O$, $C_O(G_{\chi})=Z(G_{\chi})$, $O/Z(G_{\chi})=\Aut(G_{\chi})_{\bar{\chi}}$ and $O/G_{\chi}$ is trivial or cyclic of prime order. The end of the proof of lemma 8.1 of \cite{Spa} shows that there exists a projective representation $\mathscr{P}$ of $\Aut(G_{\chi})_{\bar{\chi}}$ obtained from $\bar{\chi}$ using $\bar{rep}$ and a projective representation of $\Aut(G_{\chi})_{D.K_{\chi}/K_{\chi}, \bar{\chi}}$ obtained from $\bar{\Omega}_D(\chi)$ using $\bar{rep}$ (note that $\bar{\Omega}_D(\chi)$ is well defined thanks to condition $2$ in definition \ref{WeakAM}). But $\Aut(G_{\chi})_{\bar{\chi}}=\Aut(S)_{\chi}$ and $\Aut(G_{\chi})_{D.K_{\chi}/K_{\chi}, \bar{\chi}}=\Aut(S)_{D,\chi}$ (see \cite{GLS3}, corollary 5.1.4) and because of our choice of $\bar{rep}$ we get projective representations obtained from $\chi$ and $\Omega_D(\chi)$ using $rep$. Those projective representations are according to proof of lemma 8.1 of \cite{Spa} such that the conditions 7.2(iii) in \cite{Spa} are satisfied.
\end{proof}

Now let us give a lemma that will simplify the checking of condition $2$ of definition \ref{WeakAM} :

\begin{lem}\label{kernelheight}
 Let $p$ be any prime number. Let $S$ be a perfect simple group, $G$ its universal covering group and $Z=Z(G)$. Let $D$ be a Sylow $p$-subgroup of $G$. Suppose that :
  \begin{enumerate}
  \item $Z(G)_{p}$ is simple and $D$ is not abelian
  \item there exists an $\Aut(G)_D$-equivariant bijection $\Omega_D : \Irr_0(G|D) \rightarrow \Irr_0(N_G(D)|D)$
  \item $\Omega_D(\Irr_0(B))=\Irr_0(b)$ whenever $B$ and $b$ are Brauer correspondent
  \end{enumerate}
Then the weak AM condition is satisfied for $S$ with respect to $D$.
\end{lem}
\begin{proof}
  It suffices to show that all height zero characters of $\Irr_0(G|D)$ and of $\Irr_0(N_G(D)|D)$ have the $p$-part of the center denoted by $Z(G)_p$ in its kernel. But this is a direct consequence of the hypothesis, of lemma \ref{simplelem} and the fact that $G$ is perfect.
\end{proof}

\subsection{Inductive AM condition for $\mathfrak{A}_7$}

Using \cite{GAP4} we have the following :

\begin{lem}
 The group $6.\mathfrak{A}_7$ has $8$ blocks. Four of them have central defect. Three of them denoted by $B_1,B_2,B_3$ have maximal defect (and numerical defect $4$). The remaining one denoted by $B_4$ has numerical defect $3$.
\end{lem}

We first deal with blocks of maximal defect. Let $D_1$ be a Sylow subgroup of $6.\mathfrak{A}_7$ and define $A_1=\Aut(6.\mathfrak{A}_7)_{D_1}$. In view of lemmas \ref{kernelheight} and \ref{lemA6A7} and the fact that $\Out(\mathfrak{A}_7)$ is cyclic, we only have to find an $A_1$-equivariant bijection preserving Brauer corresponding blocks.

\begin{lem}
For all $i\in \{1,2,3\}$, $|\Irr_0(B_i)|=4$ and there exists a labelling of the elements of $\Irr_0(B_i)=\{X^i_1,X^i_2,X^i_3,X^i_4\}$ such that :
\begin{itemize}
\item for all $j\in \{1,2,3,4\}$, $(X^2_j)^{A_1}=\{X^2_j,X^{3}_j\}$ 
\item all other characters of $B_1,B_2,B_3$ are $A_1$-invariant
\end{itemize}
\end{lem}
\begin{proof}
  This is a straightforward calculation using \cite{GAP4} and more precisely the \cite{AtlasRep} package and functions such as PrimeBlocks and BlocksInfo.
\end{proof}
\begin{lem}
  Keep the notations of the previous lemma. Then the group $N_{6.\mathfrak{A}7}(D_1)$ has $3$ blocks of maximal defect denoted by $b_1, b_2,b_3$ and labelled such that $b_i$ is the Brauer correspondent of $B_i$. Moreover for all $i\in \{1,2,3\}$, there exists a labelling of $\Irr_0(b_i)=\{x^i_1,x^i_2,x^i_3,x^i_4\}$ such that :
\begin{itemize}
\item for all $j\in\{1,2,3,4\}$, $(x^2_j)^{A_1}=\{x^2_j,x^{3}_j\}$
\item all other characters are $A_1$-invariant
\end{itemize}
\end{lem}
\begin{proof}
  This is again straightforward calculation using \cite{GAP4}.
\end{proof}

These lemmas allow us to build an $A_1$-equivariant bijection $\Omega_{D_1} : \Irr_0(6.\mathfrak{A}_7|D_1) \rightarrow \Irr_0(N_{6.\mathfrak{A}_7}(D_1)|D_1), X^i_j\mapsto x_j^i$. Then the condition of lemma \ref{lemA6A7} are satisfied and the inductive AM condition is satisfied for $\mathfrak{A}_7$ with respect to $D_1$.

To deal with the last block we need the following :

\begin{lem}
  Let $D_2$ be the defect group of $B_4$ and define $A_2=\Aut(6.\mathfrak{A}_7)_{D_2}$. Let $b_4$ be the Brauer correspondent of $B_4$ in $N_{6.\mathfrak{A}_7}(D_2)$ and let $\mu$ be the trivial character of $Z=Z(6.\mathfrak{A}_7)$. Then :
  \begin{itemize}
  \item $|\Irr_0(B_4)|=4$ and all height zero characters are $A_2$-invariant
  \item $\forall \chi \in \Irr_0(B_4), \chi_{|Z}=\chi(1).\mu$
  \item $|\Irr_0(b_4)|=4$ and all height zero characters are $A_2$-invariant
  \item $\forall \chi \in \Irr_0(b_4), \chi_{|Z}=\chi(1).\mu$
\end{itemize}
\end{lem}
\begin{proof}
This is again proved using \cite{GAP4}. Note that $N_{6.\mathfrak{A}_7}(D_2)$ is found as the only normalizer of a subgroup of index two of a Sylow subgroup of $6.\mathfrak{A}_7$ which has a block of numerical defect $3$. This block is necessarily the block $b_4$.
\end{proof}

Thus we can produce an $A_2$-equivariant map $\Omega_{D_2} : \Irr_0(6.\mathfrak{A}_7|D_2) \rightarrow \Irr_0(N_{6.\mathfrak{A}_7}(D_2)|D_2)$ so that the weak inductive AM condition is true for $\mathfrak{A}_7$ with respect to $D_2$. Then as before the condition of lemma \ref{lemA6A7} is trivially satisfied for all elements of $\Irr_0(6.\mathfrak{A}_7|D_2)$. 

Everything put together gives :

\begin{prop}\label{indAM7}
  The inductive AM condition holds for $\mathfrak{A}_7$.
\end{prop}

\subsection{Inductive AM condition for $\mathfrak{A}_6$}

We know by \cite{Fry}, theorem 5.8, that :
\begin{prop}\label{indAM6}
  The inductive AM condition holds for $\mathfrak{A}_6$.
\end{prop}
\begin{rmk}
  All the blocks have central defect groups or maximal defect groups. The proof can be done using lemma \ref{lemA6A7} for all characters but one that we denote by $\chi$ lying in the principal block which is stable under the whole automorphism group. Let $D$ be a Sylow subgroup. One can check that in $6.\mathfrak{A}_6$ and in $N_{6.\mathfrak{A}_6}(D)$, $\chi$ and its image under a constructed equivariant bijection are in fact restrictions of linear characters of the relevant automorphisms groups and that $\Aut(\mathfrak{A}_6)_D$ is a $2$-group so that the conditions 7.2(iii) are satisfied. Moreover the characters of height zero have $Z(6.\mathfrak{A}_6)_2$ in their kernel because all the blocks we are dealing with are of maximal defect so that lemma \ref{simplelem} applies and the condition $2$ of definition \ref{WeakAM} is satisfied.
\end{rmk}

This concludes the proof of theorem \ref{InductiveAM}.

\bibliographystyle{alpha}
\bibliography{bibio}

\noindent\textsc{Universit\'e Paris-Diderot, Institut Math\'ematique de Jussieu - Paris Rive Gauche, B\^atiment Sophie Germain, 75205 Paris Cedex 13, France.}\\
\texttt{E-mail : denoncin@math.jussieu.fr}

\end{document}